\documentclass[12pt,reqno]{amsart}
\usepackage{amssymb, amsmath,amsthm}
\usepackage{graphicx}
\newtheorem{thm}{Theorem}

\newtheorem{defn}{Definition}

\numberwithin{defn}{section}
\numberwithin{thm}{section}
\numberwithin{Lemma}{section}
\numberwithin{Corollary}{section}
\numberwithin{Example}{section}
\numberwithin{subsection}{section}
\numberwithin{Remark}{section}
\numberwithin{equation}{section}
\numberwithin{ppn}{section}
\begin{document}
\title[ Some Classes of third and Fourth-order iterative methods ... ]
{Some Classes of third and Fourth-order iterative methods for solving nonlinear equations} 
\author{J. P. Jaiswal }
\date{}
\maketitle


\textbf{Abstract.} 
The object of the present work is to present the new classes of third-order and fourth-order iterative methods for solving nonlinear equations. Our third-order method includes methods of Weerakoon \cite{Weerakoon}, Homeier  \cite{Homeier2}, Chun \cite{Chun} e.t.c. as particular cases. After that we make this third-order method to  fourth-order (optimal) by using a single weight function rather than using two different weight functions in  \cite{Soleymani}. Finally some examples are given to illustrate the performance of the our method by comparing with new existing third and fourth-order methods.\\

\\


\section{Introduction}
Solving nonlinear equations is one of the most important problems in numerical analysis. To solve nonlinear equations,
iterative methods such as Newton's method are usually used. Throughout this paper we consider iterative methods to find a simple root $\alpha$,  of a nonlinear equation $f(x)=0$, where $f:I\subset R \rightarrow R$ for an open interval $I$. It is known that the order of convergence of the Newton's method is two. To improve the order of convergence and efficiency index many modified third-order methods have been presented in the literature by using different techniques. Such as Weerakoon et. al. in \cite{Weerakoon} obtained a third-order method by approximating the integral in Newton's theorem by trapezoidal rule; Homeier in \cite{Homeier2} by using inverse function theorem and Chun et. al. in \cite{Chun} by using circle of curvature concept e.t.c.. Kung and Traub \cite{Kung} presented a hypothesis on the optimality of the iterative method by giving $2^{n-1}$ as the optimal order. This means that the Newton iteration by two evaluations per iterations is optimal with 1.414 as the efficiency index. By taking into account the optimality concept, many authors have tried to build iterative methods of optimal higher order of convergence.

The order of all these above discussed methods are three with three (one derivative and two function) function evaluations per full iteration. Clearly its efficiency index $(3^{1/3}\approx 1.442)$ is not high (optimal). In recent days authors are modifying  these type of non-optimal order methods to optimal order by using different techniques, such as in \cite{Behl} by using linear combination of two third-order methods,  in \cite{Torres} by using p0lynomoial approximation e.t.c.. Recently Soleymani et. al. \cite{Soleymani} have used two different weight functions in Weerakoon \cite{Weerakoon} and Homier \cite{Homeier2} methods to make it optimal.
 
This paper is organized as follows: in section 2, we describe a new class of third-order iterative method by using the concept of weight function which includes the methods of \cite{Weerakoon}, \cite{Homeier2} and  \cite{Chun} e.t.c..  In the next section we optimize the methods of previous section by again using the same weight function. Finally in the last section we give some numerical examples and the new methods are compared in the performance with some new existing third and fourth-order methods.

                             
\section{ Methods and convergence analysis}
Before constructing the methods, here  we state the following  definitions:
\begin{defn}
Let f(x) be a real valued function with a simple root $\alpha$ and let ${x_n}$ be a sequence of real numbers that converge towards $\alpha$. The order of convergence m is given by
\begin{equation}\label{eqn:21}
\lim_{n\rightarrow\infty}\frac{x_{n+1}-\alpha}{(x_n-\alpha)^m}=\zeta\neq0,  
\end{equation}     
\noindent
where $\zeta$ is the asymptotic error constant and $m \in R^+$.\\
\end{defn}

\begin{defn}
Let $n$ be the number of function evaluations of the new method. The efficiency of the new method is measured by the concept  of efficiency index \cite{Gautschi,Traub1} and defined as
\begin{equation}\label{eqn:23}
m^{1/n},
\end{equation}
where $m$ is the order of convergence of the new method.\\
\end{defn}

\subsection{ Third-order Methods }
In this section  we construct a class two-step third-order iterative method. Let us consider the following iterative formula

\begin{eqnarray}\label{eqn:21}
y_n&=&x_n-\frac{f(x_n)}{f'(x_n)},\nonumber\\
x_{n+1}&=&x_n-A \left(t\right)\frac{f(x_n)}{f'(x_n)},
\end{eqnarray}
where $t=\frac{f'(y_n)}{f'(x_n)}$. The following theorem indicates under what conditions on the weight functions  in $(\ref{eqn:21})$, the order of convergence is three:

\begin{thm}
Let the function f have sufficient number of continuous derivatives in a neighborhood of $\alpha$ which is a simple root of f, then the method $(\ref{eqn:21})$ has third-order convergence, when the weight function $A(t)$  satisfies the following conditions:
\begin{eqnarray}\label{eqn:22}
A(1)=1,\  A^{'}(1)=-\frac{1}{2},\  \left|A^{''}(1)\right|\leq +\infty,
\end{eqnarray}
\end{thm}
\begin{proof}
Let $e_n=x_n-\alpha$ be the error in the $n^{th}$ iterate and $c_h=\frac{f^{(h)}(\alpha)}{h!}$, $h=1,2,3 . . .$. We provide the Taylor series expansion of each term involved in $(\ref{eqn:21})$. By Taylor expansion around the simple root in the $n^{th}$ iteration, we have\\
\begin{equation}\label{eqn:23}
\begin{split}
f(x_n)=f'(\alpha)[e_n+c_2e_n^2+c_3e_n^3+c_4e_n^4+c_5e_n^5]         
\end{split}
\end{equation}
and, we have
\begin{equation}\label{eqn:24}
\begin{split}
f'(x_n)=f'(\alpha)[1+2c_2e_n+3c_3e_n^2+4c_4e_n^3+5c_5e_n^4].           
\end{split}
\end{equation}

Further more it can be easily find 
\begin{equation}\label{eqn:25}
\frac{f(x_n}{f'(x_n)}=e_n-c_2e_n^2+(2c_2^2-2c_3)e_n^3+O(e_n^4).
\end{equation}

By considering this relation, we obtain
\begin{equation}\label{eqn:26}
y_n=\alpha+c_2e_n^2+2(c_3-c_2^2)e_n^3+O(e_n^{4}).
\end{equation}

At this time, we should expand $f'(y_n)$ around the root by taking into consideration $(\ref{eqn:26})$. Accordingly, we have
\begin{equation}\label{eqn:27}
f'(y_n)=f'(\alpha)[1+2c_2^2e_n^2+(4c_2c_3-4c_2^3)e_n^3+O(e_n^{4})].
\end{equation}
Furthermore, we have
\begin{eqnarray}\label{eqn:28}
\frac{f'(y_n)}{f'(x_n)}=1-2c_2e_n+\left(6c_2^2 - 3c_3\right)e_n^2+ . . . +O(e_n^{4}).
\end{eqnarray}

By virtue of $(\ref{eqn:28})$ and $(\ref{eqn:22})$, we attain
\begin{eqnarray}\label{eqn:29}
A(t)\times \frac{f(x_n)}{f'(x_n)}=e_n-\frac{1}{2}\left[c_3-4c_2^2(-1+A''(1))\right]e_n^3 +O(e_n^{4}).
\end{eqnarray}
Finally using $(\ref{eqn:29})$ in $(\ref{eqn:21})$, we can have the following general equation, which has the third-order convergence 
\begin{eqnarray}\label{eqn:210}
&& e_{n+1}=x_{n+1}-\alpha \nonumber\\
&& =x_n-A(t)\times \frac{f(x_n)}{f'(x_n)}-\alpha \nonumber\\
&&=\frac{1}{2}\left[c_3-4c_2^2(-1+A''(1))\right]e_n^3 +O(e_n^{4}).
\end{eqnarray}
This proves the theorem.

\end{proof}

\textbf{Particular Cases:}\\

\textit{Case 1:} If we take $A(t)=\frac{2}{1+t}$ in $(\ref{eqn:21})$, then we get the  formula
\begin{eqnarray}\label{eqn:211}
y_n&=&x_n-\frac{f(x_n)}{f'(x_n)},\nonumber\\
x_{n+1}&=&x_n-\frac{2f(x_n)}{f'(x_n)+f'(y_n)},
\end{eqnarray}
which is same as established by Weerakoon et. al. in \cite{Weerakoon}.
\\

\textit{Case 2:}
If we take $A(t)=\frac{t+1}{2t}$ in $(\ref{eqn:21})$, then we get the  formula
\begin{eqnarray}\label{eqn:212}
y_n&=&x_n-\frac{f(x_n)}{f'(x_n)},\nonumber\\
x_{n+1}&=&x_n-\frac{f(x_n)}{2}\left(\frac{1}{f'(x_n)}+\frac{1}{f'(y_n)}\right),
\end{eqnarray}
which is same as established by Homeier in \cite{Homeier2}.
\\

\textit{Case 3:}
If we take $A(t)=\frac{3-t}{2}$ in $(\ref{eqn:21})$, then we get the  formula
\begin{eqnarray}\label{eqn:213}
y_n&=&x_n-\frac{f(x_n)}{f'(x_n)},\nonumber\\
x_{n+1}&=&x_n-\frac{1}{2}\left(3-\frac{f'(y_n)}{f'(x_n)}\right)\frac{f(x_n)}{f'(x_n)},
\end{eqnarray}
which is same as established by Chun et. al. in \cite{Chun}.
\\

\textit{Case 4:}
If we take $A(t)=\frac{3-t}{2}+\gamma (t-1)^2$ in $(\ref{eqn:21})$, where $\gamma$ is a real constant then we get the formula
\begin{eqnarray}\label{eqn:214}
y_n=x_n&-&\frac{f(x_n)}{f'(x_n)},\nonumber\\
x_{n+1}=x_n&-& \left[\frac{3}{2}-\frac{f'(y_n)}{2f'(x_n)}+\gamma \left(\frac{f'(y_n)}{f'(x_n)}-1\right)^2\right]\frac{f(x_n)}{f'(x_n)},
\end{eqnarray}
and its error equation is
\begin{eqnarray}\label{eqn:215}
e_{n+1}=\frac{1}{2}\left[(4-8\gamma)c_2^2+c_3\right]e_n^3 +O(e_n^{4}).
\end{eqnarray}
\textit{Remark 1.} Since $\gamma \in R$, by varying it one can get infinite number of third-order methods.


\subsection{ Fourth-order Methods }

The convergence order of the previous section methods are three with three (one derivative and two function) function evaluations per full iteration. Clearly its efficiency index $(3^{1/3}\approx 1.442)$ is not high (optimal). We now make use of one more same weight function to build our optimal class based on $(\ref{eqn:21})$ by a simple change in its first step.  Thus we consider
\begin{eqnarray}\label{eqn:31}
y_n&=&x_n-a.\frac{f(x_n)}{f'(x_n)},\nonumber\\
x_{n+1}&=&x_n-\{P(t)\times Q(t)\}\frac{f(x_n)}{f'(x_n)},
\end{eqnarray}
where  $P(t)$ and $Q(t)$ is a real-valued weight function with  $t=\frac{f'(y_n)}{f'(x_n)}$ and $a$ is a real constant. The weight function should be chosen such that order of convergence arrives at optimal level  four without using more function evaluations. The following theorem indicates under what conditions on the weight functions and constant $a$ in $(\ref{eqn:31})$, the order of convergence will arrive at the optimal level four:

\begin{thm}
Let the function f have sufficient number of continuous derivatives in a neighborhood of $\alpha$ which is a simple root of f, then the method $(\ref{eqn:31})$ has fourth-order convergence, when $a=2/3$ and the weight function $P(t)$  and and $Q(t)$ satisfy the following conditions

\begin{eqnarray}\label{eqn:31a}
&& P(1)= 1,\  P'(1) =-\frac{1}{2}, \ \left|P^{(3)}(1)\right|\leq +\infty \nonumber \\
&& Q(1) = 1,\ Q'(1)= -\frac{1}{4}, \  Q''(1) = 2 - P''(1),\  \left|Q^{(3)}(1)\right|\leq +\infty.\nonumber \\
&&
\end{eqnarray}
\end{thm}

\begin{proof}
Using $(\ref{eqn:23})$ and $(\ref{eqn:24})$ and $a=2/3$ in the first step of $(\ref{eqn:31})$, we have
\begin{eqnarray}\label{eqn:32}
y_n=\alpha+ \frac{e_n}{3}+\frac{2c_2e_n^2}{3}+\frac{4(c_3-c_2^2)e_n^3}{3} + . . . +O(e_n^5).
\end{eqnarray}
Now we should expand $f'(y_n)$ around the root by taking into consideration $(\ref{eqn:32})$.Thus, we have
\begin{equation}\label{eqn:33}
f'(y_n)=f'(\alpha)\left[1+\frac{2c_2e_n}{3}+\frac{(4c_2^2+c_3)e_n^2}{3}+ . . . +O(e_n^{5})\right].
\end{equation}
Furthermore, we have
\begin{eqnarray}\label{eqn:34}
\frac{f'(y_n)}{f'(x_n)}=1-\frac{2c_2}{3}e_n+\left(4c_2^2 - \frac{8c_3}{3}\right)e_n^2+ . . . +O(e_n^{5}).
\end{eqnarray}
By virtue of $(\ref{eqn:34})$ and $(\ref{eqn:31a})$, we attain
\begin{eqnarray}\label{eqn:35}
&&\{P(t)\times Q(t)\}\frac{f(x_n)}{f'(x_n)} \nonumber\\
&&=e_n-\frac{1}{81}\left[-81c_2c_3+9c_4+(309+24A''(1)+32A^{(3)}(1)+32B^{(3)}(1))c_2^3\right]e_n^4 \nonumber\\
&&+O(e_n^{5})
\end{eqnarray}
Finally using $(\ref{eqn:35})$ in $(\ref{eqn:31})$, we can have the following general equation, which reveals the fourth-order convergence 
\begin{eqnarray}\label{eqn:36}
 e_{n+1}&=&x_{n+1}-\alpha \nonumber\\
&=&x_n-\{P(t)\times Q(t)\}\frac{f(x_n)}{f'(x_n)}-\alpha \nonumber\\
&=&\frac{1}{81}\left[-81c_2c_3+9c_4+(309+24A''(1)+32A^{(3)}(1)+32B^{(3)}(1))c_2^3 \right]e_n^4 \nonumber \\
&& +O(e_n^{5}).
\end{eqnarray}
It confirms the result.
\end{proof}
\textit{Remark 2.} Since $\gamma \in R$, by varying it one can get infinite number of fourth-order methods.

It is obvious that our novel class of iterations require three function evaluations per iteration, i.e. two first derivative and one function evaluations. Thus our new method is optimal. Clearly its efficiency index is $4^{1/4}=1.5874$ (high). Now by choosing appropriate weight functions as presented in $(\ref{eqn:31})$, we can give optimal two-step fourth-order iterative methods, such as
\\\\\\\\

\textit{Method. 1}

\begin{eqnarray}\label{eqn:38}
y_n=x_n&-&\frac{2}{3}\frac{f(x_n)}{f'(x_n)},\nonumber\\
x_{n+1}=x_n&-&\left[2-\frac{7}{4}\frac{f'(y_n)}{f'(x_n)}+\frac{3}{4}\left(\frac{f'(y_n)}{f'(x_n)}\right)^2\right]\frac{2f(x_n)}{f'(x_n)+f'(y_n)},
\end{eqnarray}

where its error equation is
\begin{eqnarray}\label{eqn:39}
e_{n+1}=\frac{1}{9}\left[-9c_2c_3+c_4+33c_2^3\right]e_n^4 +O(e_n^{5}).
\end{eqnarray} 
 
 \textit{Method. 2}

\begin{eqnarray}\label{eqn:310}
y_n=x_n&-&\frac{2}{3}\frac{f(x_n)}{f'(x_n)},\nonumber\\
x_{n+1}=x_n&-&\left[\frac{7}{4}-\frac{5}{4}\frac{f'(y_n)}{f'(x_n)}+\frac{1}{2}\left(\frac{f'(y_n)}{f'(x_n)}\right)^2\right]\nonumber\\
&&\times \frac{f(x_n)}{2}\left(\frac{1}{f'(x_n)}+\frac{1}{f'(y_n)}\right),
\end{eqnarray}

where its error equation is
\begin{eqnarray}\label{eqn:311}
e_{n+1}=\frac{1}{9}\left[-c_2c_3+\frac{c_4}{9}+\frac{79}{27}c_2^3\right]e_n^4 +O(e_n^{5}).
\end{eqnarray} 

\textit{Method. 3}

\begin{eqnarray}\label{eqn:312}
y_n=x_n&-&\frac{2}{3}\frac{f(x_n)}{f'(x_n)},\nonumber\\
x_{n+1}=x_n&-&\left[\frac{9}{4}-\frac{9}{4}\frac{f'(y_n)}{f'(x_n)}+\left(\frac{f'(y_n)}{f'(x_n)}\right)^2\right]\nonumber\\
&& \times \left(\frac{3}{2}-\frac{f'(y_n)}{2f'(x_n)}\right)\frac{f(x_n)}{f'(x_n)},
\end{eqnarray}
which is the same as given in  \cite{Jaiswal}
and its error equation is
\begin{eqnarray}\label{eqn:313}
e_{n+1}=\frac{1}{9}\left[-c_2c_3+\frac{c_4}{9}+\frac{103}{27}c_2^3\right]e_n^4 +O(e_n^{5}).
\end{eqnarray}

\textit{Method 4:}
\begin{eqnarray}\label{eqn:314}
y_n=x_n&-&\frac{2}{3}\frac{f(x_n)}{f'(x_n)},\nonumber\\
x_{n+1}=x_n&-&\left[\left(\frac{9}{4}-\gamma\right)
+\left(2\gamma-\frac{9}{4}\right)\frac{f'(y_n)}{f'(x_n)}+(1-\gamma)\left(\frac{f'(y_n)}{f'(x_n)}\right)^2\right]\nonumber\\
&&\times \left[\frac{3}{2}-\frac{f'(y_n)}{2f'(x_n)}+\left(\frac{f'(y_n)}{f'(x_n)}-1\right)^2\right]\frac{f(x_n)}{f'(x_n)},
\end{eqnarray}
and its error equation is
\begin{eqnarray}\label{eqn:315}
e_{n+1}=\frac{1}{27}\left[-27c_2c_3+3c_4+(103+16 \gamma)c_2^3\right]e_n^4 +O(e_n^{5}).
\end{eqnarray}

\section{Numerical Testing}   

Here we consider, the following four test functions to illustrate the accuracy of new iterative methods. The root of each nonlinear test function is also listed. All the computations reported here we have done using Mathematica 8. Scientific computations in many branches of science and technology demand very high precision degree of numerical precision.  We consider the number of decimal places as follows: 200 digits floating point (SetAccuraccy=200) with  SetAccuraccy Command.  The test non-linear functions are listed in Table-1.

Here we comparer performance of our new method $(\ref{eqn:38})$ to the methods of Weerakoon  $(\ref{eqn:211})$, Homeier $(\ref{eqn:212})$, Chun $(\ref{eqn:213})$, method  (17) of \cite{Soleymani}, Khattri \cite{Khattri}. The results of comparison for the test function are provided in the table 2-5. 


\begin{table}[htb]
\caption{Functions and Roots.}
  \begin{tabular}{ll} \hline
$f(x)$                                           &$\alpha$  \\ \hline 
$f_1(x)=e^{-x}-1+\frac{x}{5}$                       &4.9651142317442763036...\\ 
$f_2(x)=\frac{x^3+2.87x^2-10.28}{4.62}-x$           &2.0021187789538272889...\\ 
$f_3(x)=\frac{x+Cos x\ Sin x }{\pi}-\frac{1}{4}$    &0.4158555967898679887... \\  
$f_4(x)=xe^{-x}-0.1$                                &0.1118325591589629648... \\                                                            
\hline
  \end{tabular}
  \label{tab:abbr}
\end{table}
                

\begin{table}[htb]
\caption{Errors Occurring in the estimates of the root of function $f_1(x)$ by the methods described below with initial guess $x_0=5$.}
  \begin{tabular}{llll} \hline
Methods                         &$\left|x_1-\alpha \right|$  &$\left|x_2-\alpha \right|$ & $\left|x_3-\alpha \right|$ \\ \hline 
Newton Method                   & 0.21464e-4 & 0.83264e-11 &0.12530e-23\\ 
Weerakoon   $(\ref{eqn:211})$   & 0.11208e-6 & 0.37810e-23 &0.14517e-72\\ 
Homeier     $(\ref{eqn:212})$   & 0.12544e-6 & 0.59456e-23 &0.63310e-72 \\  
Chun Method $(\ref{eqn:213})$   & 0.98734e-7 & 0.22705e-23 &0.27611e-73 \\ 
Method  (17) of \cite{Soleymani}   
            & 0.14780e-5        &0.47702e-23 & 0.51761e-93 \\ 
Method of \cite{Khattri}   
            & 0.47426e-9        &0.16796e-40 & 0.26418e-166 \\
Method      $(\ref{eqn:38})$    & 0.42743e-9 & 0.99425e-41 &0.29108e-167\\                                                                
\hline
  \end{tabular}
  \label{tab:abbr}
\end{table}

\newpage

\begin{table}[htb]
 \caption{Errors Occurring in the estimates of the root of function $f_2(x)$ by the methods described below with initial guess $x_0=2.5$.}
  \begin{tabular}{llll} \hline
Methods                         &$\left|x_1-\alpha \right|$  &$\left|x_2-\alpha \right|$ & $\left|x_3-\alpha \right|$ \\ \hline 
Newton Method                   & 0.85925e-1 & 0.32675e-2 &0.50032e-5\\ 
Weerakoon   $(\ref{eqn:211})$   & 0.18271e-1 & 0.14770e-5 &0.79610e-18\\ 
Homeier     $(\ref{eqn:212})$   & 0.49772e-2 &0.33027e-8 &0.95318e-27 \\ 
Chun Method $(\ref{eqn:213})$   & 0.27815e-1 & 0.95903e-5 &0.41254e-15 \\ 
Method (17) of \cite{Soleymani}   
                               & 0.26594e-1  &0.32982e-6  & 0.76311e-26 \\ 
Method of \cite{Khattri}   
                               & 0.14965e-1  &0.45484e-7 & 0.40826e-29 \\
Method      $(\ref{eqn:38})$    & 0.76770e-2 & 0.12105e-8 &0.76261e-36\\                                                 
\hline
  \end{tabular}
  \label{tab:abbr}
\end{table}

\newpage

\begin{table}[htb]
 \caption{Errors Occurring in the estimates of the root of function $f_3(x)$ by the methods described below with initial guess $x_0=0.4$.}
  \begin{tabular}{llll} \hline
Methods                         &$\left|x_1-\alpha \right|$  &$\left|x_2-\alpha \right|$ & $\left|x_3-\alpha \right|$ \\ \hline 
Newton Method                   & 0.10737e-3 & 0.50901e-8 &0.11442e-16\\ 
Weerakoon   $(\ref{eqn:211})$   & 0.20631e-6 & 0.53436e-21 &0.92858e-65\\ 
Homeier     $(\ref{eqn:212})$   & 0.52795e-6 & 0.19743e-19 &0.10325e-59 \\ 
Chun Method $(\ref{eqn:213})$   & 0.93064e-6 & 0.20624e-18 &0.22446e-56 \\   
Method (17) of \cite{Soleymani}   
            & 0.86290e-7        &0.78612e-28  & 0.54150e-112 \\  
Method of \cite{Khattri}   
           & 0.52074e-7         &0.67319e-29 &   0.18803e-116 \\  
Method      $(\ref{eqn:38})$    & 0.24363e-7 & 0.14724e-30 &0.19642e-123\\           
\hline
  \end{tabular}
  \label{tab:abbr}
\end{table}


\begin{table}[htb]
 \caption{Errors Occurring in the estimates of the root of function $f_4(x)$ by the methods described below with initial guess $x_0=0.3$.}
  \begin{tabular}{llll} \hline
Methods                         &$\left|x_1-\alpha \right|$  &$\left|x_2-\alpha \right|$ & $\left|x_3-\alpha \right|$ \\ \hline 
Newton Method                   & 0.47567e-1 & 0.22849e-2 &0.55356e-5\\ 
Weerakoon   $(\ref{eqn:211})$   & 0.13039e-1 & 0.31800e-5 &0.45048e-16\\ 
Homeier     $(\ref{eqn:212})$   & 0.64393e-3 & 0.72236e-10 &0.10226e-30 \\ 
Chun Method $(\ref{eqn:213})$   & 0.34012e-1 & 0.11125e-3 &0.34855e-11 \\   
Method (17) of \cite{Soleymani}   
            & 0.14507e-1        &0.20259e-6  & 0.81662e-26 \\  
Method of \cite{Khattri}   
           & 0.56586e-1         &0.78841e-4&   0.41658e-15 \\  
Method      $(\ref{eqn:38})$    & 0.11886e-1 & 0.73037e-7 &0.10950e-27\\           
\hline
  \end{tabular}
  \label{tab:abbr}
\end{table}

\section{Conclusion}
In this paper we gave a class of the third-order iterative method which includes the methods of Weerakoon \cite{Weerakoon}, Homeier \cite{Homeier2} and  Chun \cite{Chun} as  particular cases. We also constructed a class of the optimal fourth-order method. This method gives different fourth-order methods by slight changing in first-step and using a single weight function in second step of Weerakoon \cite{Weerakoon} , Homeier \cite{Homeier2} and  Chun \cite{Chun} method  rather than using two different weight functions in  \cite{Soleymani}. A number of examples are given to illustrate the performance of our method by comparing with new existing third and fourth-order methods.

\textsc{Jai Prakash Jaiswal\\
Department of Mathematics, \\
Maulana Azad National Institute of Technology,\\
Bhopal, M.P., India-462051}.\\
E-mail: {asstprofjpmanit@gmail.com; jaiprakashjaiswal@manit.ac.in}.\\\\
\end{document}